\documentclass{amsart}
\usepackage{amsmath,amssymb,amscd,amsfonts,amsthm,mathrsfs, graphicx}
\usepackage[arrow,matrix,curve,cmtip,ps]{xy}
\usepackage{enumitem}
\usepackage{pinlabel}
\usepackage{tikz-cd}
\usepackage{mathtools}
\usepackage{color}
\usepackage{xcolor}
\usepackage{caption}
\usepackage{subcaption}
\usepackage{soul}
\usepackage{svg}
\usepackage{url}
\usepackage{CJKutf8}

\pdfoutput=1
\usepackage{geometry}
\usepackage{cancel}

\usepackage{setspace}

\usepackage{relsize}

\usepackage{centernot}

\usepackage[pagebackref=true]{hyperref}
\hypersetup{colorlinks,linkcolor={blue},citecolor={black}}

\usepackage{indentfirst}

\usepackage{pinlabel}	

\newtheorem{proposition}{Proposition}[section]
\newtheorem{theorem}[proposition]{Theorem}
\newtheorem{corollary}[proposition]{Corollary}
\newtheorem{lemma}[proposition]{Lemma}

\newtheorem*{theorem*}{Theorem}
\newtheorem*{proposition*}{Proposition}
\newtheorem*{lemma*}{Lemma}
\newtheorem*{corollary*}{Corollary}
\newtheorem{question*}{Question}
\newtheorem*{rep@theorem}{\rep@title}
\newcommand{\newreptheorem}[2]{
\newenvironment{rep#1}[1]{
 \def\rep@title{#2 \ref{##1}}
 \begin{rep@theorem}}
 {\end{rep@theorem}}}
\makeatother

\newtheorem*{rep@proposition}{\rep@title}
\newcommand{\newrepproposition}[2]{
\newenvironment{rep#1}[1]{
 \def\rep@title{#2 \ref{##1}}
 \begin{rep@proposition}}
 {\end{rep@proposition}}}
\makeatother

\theoremstyle{definition}
\newtheorem{definition}[proposition]{Definition}
\newtheorem{question}[proposition]{Question}
\newtheorem{remark}[proposition]{Remark}
\newtheorem{example}[proposition]{Example}

\newreptheorem{theorem}{Theorem}
\newreptheorem{lemma}{Lemma}
\newreptheorem{proposition}{Proposition}
\newreptheorem{corollary}{Corollary}
\newreptheorem{definition}{Definition}

\newcommand{\bdry}{\partial}

\newcommand{\Q}{\mathbb{Q}}

\newcommand{\Z}{\mathbb{Z}}

\newcommand{\C}{\mathcal{C}}
\newcommand{\CC}{\mathbb{C}}

\newcommand{\A}{\mathcal{A}}

\newcommand{\AC}{\mathcal{AC}}

\newcommand{\rk}{\operatorname{rk}}

\newcommand{\Bl}{\mathcal{B}\ell}
\newcommand{\Qt}{{\Q[t^{\pm 1}]}}
\newcommand{\Zt}{{\Z[t^{\pm 1}]}}

\begin{document}
\title[Rational Concordance of Double Twist Knots]{Rational Concordance of Double Twist Knots}

\author{Jaewon Lee}
\address{Department of Mathematical Sciences, Korea Advanced Institute for Science and Technology}
\email{freejw@kaist.ac.kr}

\date{\today}

\def\subjclassname{\textup{2020} Mathematics Subject Classification}
\expandafter\let\csname subjclassname@1991\endcsname=\subjclassname
\subjclass{57K10}

\begin{abstract} 
  Double twist knots $K_{m, n}$ are known to be rationally slice if $mn = 0$, $n = -m\pm 1$, or $n = -m$. In this paper, we prove the converse. It is done by showing that infinitely many prime power-fold cyclic branched covers of the other cases do not bound a rational ball. Our rational ball obstruction is based on Donaldson's diagonalization theorem.
\end{abstract}

\maketitle
\section{Introduction}
Two knots in $S^3$ are called \textit{concordant} if they cobound a smoothly and properly embedded annulus in $S^3\times I$. If a knot $K$ is concordant to the unknot, then $K$ is said to be \textit{slice}, meaning it bounds a smooth disk in $B^4$. The concordance relation with the connected sum operation gives rise to an abelian group $\C$, called the \textit{knot concordance group}, where the inverse of $K$ is represented by the orientation-reversed mirror image of $K$. Knot concordance has been extensively studied with deep interactions in 3- and 4-dimensional topology since Fox and Milnor initiated.

Levine \cite{Lev69} developed a simple algebraic tool to study knot concordance by using a bilinear form, called the \textit{Seifert form}, obtained from a Seifert surface. If a knot has a metabolic Seifert form, it is called \textit{algebraically slice}. Since every slice knot has a metabolic Seifert form, there is a quotient map $\C\rightarrow\AC$, where the quotient group $\AC$ is called \textit{algebraic concordance group}. Levine also proved that $\AC\cong \Z^\infty\oplus\Z_2^\infty\oplus\Z_4^\infty$. Sliceness and algebraic sliceness can be similarly defined for high-dimensional knots. In fact, it turned out that every even dimensional knot is slice \cite{Ker65}, and every higher odd dimensional knot that is algebraically slice is slice \cite{Lev69}.

It was not known whether there exists an algebraically slice knot that is not slice in the classical dimension until Casson and Gordon \cite{CG78} found such knots. Their examples are some \textit{twist knots} $K_n$ described in Figure \ref{fig-Kn} with $n\in \Z$ full-twists. More precisely, while it was known that $K_n$ is algebraically slice if and only if $n = k(k-1)$ for some integer $k$ by Levine \cite{Lev69}, Casson and Gordon \cite{CG78} proved that the only slice twist knots are $K_0$ and $K_2$, namely the unknot and the stevedore knot.

\begin{figure}[h]
  \includesvg{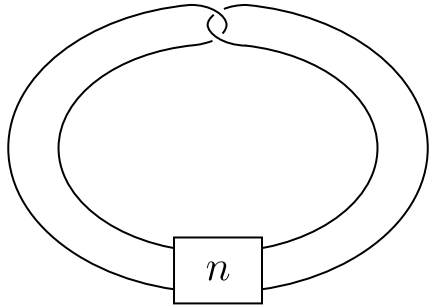}
  \caption{The twist knot $K_n$.}
\label{fig-Kn}
\end{figure}

On the other hand, there is a generalized notion of sliceness by replacing $B^4$ with a smooth rational ball. A knot which bounds a smooth disk in some rational ball is called \textit{rationally slice}. Modulo rational sliceness, we obtain the quotient group $\C_\Q$ of $\C$, called the \textit{rational knot concordance group}. By definition, every slice knot is rationally slice, but the converse is not true. Based on \cite{FS84}, Cochran proved that the figure-eight knot $K_1$, which is not even algebraically slice, is rationally slice. Thus, it is natural to ask which twist knots are rationally slice.

When $n < 0$, the twist knot $K_n$ has signature $\sigma(K_n) < 0$. Since the signature and Levine-Tristram signature function are rational concordance invariants \cite{CO93, CK02}, it is easy to see that $K_n$ for negative $n$ is not rationally slice. In fact, there are many rational concordance invariants from knot Floer homology package such as $\tau$ \cite{OS03a, Ras03}, $\epsilon$ \cite {Hom14}, $\nu^+$ \cite{HW16}, and $\Upsilon_K$ \cite{OSS17}. When $n > 0$, however, all the mentioned invariants vanish because $K_n$ is \textit{$0$-bipolar}, i.e., it bounds a smooth nullhomologous disk in both positive and negative definite simply-connected smooth 4-manifolds \cite{CHH13, CK21}.

On the other hand, as a rational analogue of $\AC$, Cha \cite{Cha07} defined the \textit{rational algebraic concordance group} $\AC_\Q$, which is a quotient group of $\C_\Q$. It was previously known that some $K_n$ for $n>0$ are not rationally slice \cite{Cha07, BD12, CFHH13} by showing that they are not algebraically rationally slice. However, the rational concordance classification of all twist knots has remained unsolved.

In this paper, we answer the question by showing that the unknot $K_0$, the figure-eight knot $K_1$, and the stevedore knot $K_2$ are the only rationally slice twist knots. In fact, we determine the rational concordance of a more general family, called the \textit{double twist knots} $K_{m,n}$ in Figure \ref{fig-Kmn} with $m,n\in \Z$ full-twists. Note that the twist knot $K_n$ is the same as the double twist knot $K_{-1,n}$. When either $m$ or $n$ is $0$, $K_{m,n}$ is the unknot. It was known that $K_{m, n}$ is slice if $mn = 0$ or $|m + n| = 1$ by Siebenmann \cite{Sie75}, and Matsumoto and Yamada \cite{MY87} proved that the converse holds as well.

\begin{figure}[h]
  \includesvg{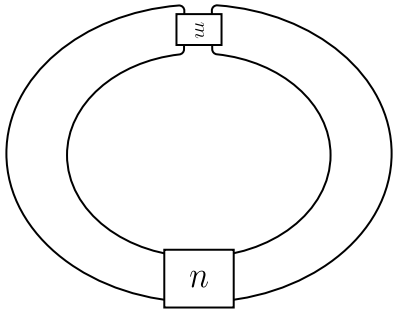}
  \caption{The double twist knot $K_{m,n}$.}
\label{fig-Kmn}
\end{figure}

\noindent It is also known that $K_{m,-m}$ is rationally slice \cite{Cha07}. We prove that $K_{m, n}$ for $mn=0$ or $|m+n| \le 1$ are the only rationally slice double twist knots:

\begin{theorem}\label{thm-A}
  The following are equivalent:
  \begin{itemize}
    \item[(a)] $K_{m,n}$ is rationally slice,
    \item[(b)] $mn = 0$ or $|m + n| \le 1$,
    \item[(c)] $K_{m,n}$ is of finite order in $\C_\Q$.
  \end{itemize}
\end{theorem}

Note that Theorem \ref{thm-A} tells us not only about the rational sliceness of $K_{m,n}$ but also about its rational concordance order in $\C_\Q$. It is not known if there exists any nontrivial torsion elements in $\C_\Q$. It was shown that some twist knots with $n>0$, whose orders in $\AC_\Q$ are $2$, generate an infinite rank subgroup in $\C_\Q$ \cite{Lee24}. We confirm that there are no nontrivial torsion elements among the double twist knots.

Note that when $m$ and $n$ have the same sign, it can be easily checked that its signature $\sigma(K_{m,n})$ is nonzero so that $K_{m,n}$ is not rationally slice. When $m$ and $n$ have opposite signs, however, it is again $0$-bipolar. Thus, as mentioned above, many rational concordance invariants vanish due to the $0$-bipolarity: \cite{CHH13, CK21}
$$\sigma(K_{m,n}) = \sigma_{K_{m,n}}(\omega) = \tau(K_{m,n})=\epsilon(K_{m,n}) = \nu^+(K_{m,n}) = \Upsilon_{K_{m,n}}(t) = 0.$$ 
Nevertheless, we prove Theorem \ref{thm-A} by employing the lattice embedding obstruction for infinitely many prime power-fold cyclic branched covers based on Donaldson's diagonalization theorem.

Recall that every prime power-fold cyclic branched cover of a slice knot bounds a smooth rational ball \cite{CG78}. To prove that a knot $K$ is not slice, it is enough to show that its double branched cover $Y$ does not bound a smooth rational ball. If $Y$ bounds a smooth definite 4-manifold $W$ and a smooth rational ball, then by Donaldson's diagonalization theorem \cite{Don87}, the intersection form $Q_W$ of $W$ must be embedded in $\langle \pm 1\rangle ^{b_2(W)}$. An embedding of a billinear form into the standard lattice of codimension $0$ is called \textit{a lattice embedding}. Lisca \cite{Lis07a} employed the lattice embedding obstruction on the double branched cover to characterize the sliceness of the $2$-bridge knots $K(p, q)$\footnote{In our notation, for example when $m < 0$ and $n > 0$, $K_{m,n} = K(4mn+1, 2n)$.} \cite{Lis07a} and their concordance order \cite{Lis07b}. See also \cite{GJ11, Lec12, Lec15, FM16, Lis17, AKPR21, Sim23, AMP24, GO25} for the lattice embedding obstructions of the double or 3-fold branched covers of other families of knots and links.

For a rational sliceness obstruction, however, the double branched cover does not suffice. For example, the double branched cover of the figure-eight knot $K_{-1, 1}$, which is rationally slice, does not bound a smooth rational ball. On the other hand, while we cannot conclude that a prime power-fold branched cover of a rationally slice knot bounds a smooth rational ball, we can conclude that such a smooth rational ball exists for sufficiently large prime power-fold branched covers:

\begin{theorem}\label{thm-B}
  If a knot $K$ is rationally slice, then for any sufficiently large prime $p$, the $p^k$-fold cyclic branched cover $\Sigma_{p^k}(K)$ bounds a smooth rational ball.
\end{theorem}

To apply Theorem \ref{thm-B} in proving Theorem \ref{thm-A}, we consider prime power-fold cyclic branched covers, denoted $Y$, of $K_{m,n}$ for infinitely many primes. Our obstruction for $Y$ bounding a smooth rational ball is essentially based on Donaldson's diagonalization theorem \cite{Don87}. Recall that the $0$-bipolarity of $K_{m,n}$ makes many rational concordance invariants vanish. Nevertheless, the $0$-bipolarity of $K_{m,n}$ helps us obtain a definite filling $W$ of $Y$ so that we can apply the lattice embedding obstruction. By analyzing a matrix which represents $Q_W$, we prove that every odd prime power-fold branched cover of $K_{m,n}$ does not bound a smooth rational ball unless $mn=0$ or $|m-n| \le 1$. We emphasize that smooth rational ball obstructions on infinitely many branched covers can distinguish rationally non-slice knots among $0$-bipolar knots.

On the other hand, for any nonzero integer $c$, any knot is $\Z\left[\frac{1}{c}\right]$-slice if and only if its $(c, 1)$-cable is $\Z\left[\frac{1}{c}\right]$-slice \cite{CFHH13}. Moreover, this fact implies that the following are equivalent:

\begin{itemize}
  \item A knot is $\Q$-slice.
  \item Its $(c, 1)$-cable for some $c$ is $\Q$-slice.
  \item Its $(c, 1)$-cables for all $c$ are $\Q$-slice.
\end{itemize}

\noindent Note that the double branched cover obstruction is a $\Z_2$-slice obstruction, which means that it obstructs $\Z\left[\frac{1}{c}\right]$-sliceness for any odd $c$. Thus, Lisca's work \cite{Lis07a} implies that the $(c, 1)$-cable of $K_{m,n}$ for each odd $c$ is not $\Z\left[\frac{1}{c}\right]$-slice unless $mn=0$ or $|m-n| = 1$. From the above fact and Theorem \ref{thm-A}, we directly obtain the following corollary about the cable of double twist knots:

\begin{corollary}\label{cor-cable}
  Any $(c, 1)$-cable of $K_{m,n}$ is not rationally slice unless $mn = 0$ or $|m+n|\le 1$.
\end{corollary}

Note that Donaldson's diagonalization theorem is a gauge theoretic result and works in the smooth category. On the other hand, one can consider the knot concordance in the topological category by allowing locally flat embedding of slice disks. Recall that the aforementioned algebraic concordance, defined by Levine \cite{Lev69}, is a topological concordance invariant. Beyond the algebraic concordance, however, the topological concordance is much different from the smooth concordance. For example, it is known that the smooth concordance group of the topologically slice knots has a $\Z_2^\infty$ subgroup \cite{HKL16} and a $\Z^\infty$ summand \cite{OSS17, DHST21}. However, it is difficult to obstruct the topological sliceness of a given knot due to the lack of tools that do not rely on smooth structures.

In fact, there are some finer topological concordance invariants than the algebraic concordance. For example, by Freedman's work \cite{Fre82, FQ90}, it turns out that the Casson-Gordon invariant \cite{CG78, CG86} still works well in the topological category. The twisted Alexander polynomial \cite{KL99} can also be used to obstruct topological sliceness. Cochran, Orr, and Teichner \cite{COT03} generalized the fact that the algebraic concordance captures whether the Alexander module admits a Lagrangian submodule with respect to the Blanchfield form. By considering higher Alexander modules and Blanchfield forms, they provided infinitely many stronger topological sliceness obstructions by employing the so-called \textit{von Neumann $\rho$-invariants}, originally defined by Cheeger and Gromov \cite{ChG85}.

Similarly, the rational knot concordance can be also studied in the topological category. Cha \cite{Cha07} introduced the notion of \textit{complexity}\footnote{Cochran, Orr, and Teichner also considered the same notion in \cite{COT03}, called \textit{multiplicity}, but they mainly studied the case when the multiplicity is $1$.} to higher Alexander modules and the corresponding Blanchfield forms to study the structure of the rational knot concordance group $\C_\Q$ in the topological category. While it is not clear that Casson-Gordon invariant and the twisted Alexander polynomial can be also used for a rational slice obstruction or has a rational concordance analogue, Cha \cite{Cha07} employed the von Neumann $\rho$-invariant as a rational sliceness obstruction. See also \cite{CHL09, Kim23, Lee24}.

In the last section, we first characterize the algebraic rational sliceness of all twist knots in Theorem \ref{thm-ACQ}, which extends the aforementioned results \cite{Cha07, BD12, CFHH13}. Then we recover a partial result of Theorem \ref{thm-A} for twist knots $K_n$ in the topological category by employing a von Neumann $\rho$-invariant with complexity.

\begin{theorem}\label{thm-C}
  $K_n$ is topologically rationally slice if and only if $ n = 0, 1,$ or $2$, provided $n\neq 2^2,3^2,\cdots,8^2$.
\end{theorem}

The $\rho$-invariant obstruction in \cite[Theorem 4.2]{COT03} typically involves infinitely many choices in practice, and the resulting set of values can be used as a sliceness obstruction, as shown in \cite[Theorem 4.6]{COT03}. To overcome the difficulty to check all such values, Cochran, Harvey, and Leidy \cite{CHL09, CHL10} developed a more practical tool. They associated a single $\rho$-invariant for each Lagrangian submodule $P$ of the Alexander module, called the \textit{first-order signature} $\rho^{(1)}(K, P)$ and proved that the set of these values suffices to obstruct sliceness \cite[Definition 4.1, Theorem 4.2]{CHL10}. In particular, the first-order signatures always form a finite set for the genus one knots.

Cha \cite{Cha07} employed a $\rho$-invariant with complexity and gave a rational sliceness obstruction in a manner similar to \cite[Theorem 4.6]{COT03}. Based on his construction, we take a Lagrangian submodule $P$ of the Alexander module with complexity $c$ and define the corresponding $\rho$-invariant, denoted $\rho^{(1)}_c(K, P)$, following \cite[Definition 4.1]{CHL10}. We then prove the following rational sliceness obstruction, which is an analogue of \cite[Theorem 4.2]{CHL10} derived from the first-order signature:

\begin{theorem}\cite{Lee24}\label{thm-D}
  If $K$ is rationally slice in a rational ball $V^4$ with complexity $c$, then there exists a Lagrangian submodule $P$ of $\A_c (K)$ such that $\rho_c^{(1)}(K, P)=0.$
For a slice disk $\Delta$, such a submodule $P$ is given by
\begin{equation*}
  P = \ker (\A_c (K)\rightarrow H_1(V-\nu\Delta;\Qt)).
\end{equation*}
\end{theorem}

This theorem was initially mentioned in \cite{Lee24} for the completeness of the preliminaries and was briefly proved, although it was not used there. In this paper, we provide a more detailed proof and use it to prove part of Theorem \ref{thm-C}.

Finally, we remark on the Dehn surgery along twist knots in the smooth category. Recall that the $0$-surgery $S^3_0(K)$ along a knot $K$ bounds a rational homology $S^1\times B^3$ if and only if $K$ is rationally slice \cite{CFHH13}. Therefore, we obtain the following immediate corollary:

\begin{corollary}
  $S^3_0(K_n)$ bounds a smooth rational homology $S^1\times B^3$ if and only if $n = 0, 1, 2$.
\end{corollary}

Similarly, the homology spheres obtained by $\pm 1$-surgery along a rationally slice knot bound a smooth rational balls. However, the converse does not hold. For example, $S^3_1(K_4)$ bounds a smooth contractible 4-manifold \cite{Fic84}, and $S^3_1(K_3)$ \cite{AL18} and $S^3_{-1}(K_n)$ for all $n$ \cite{Sim21} bound smooth rational balls. It is not known if the $1$-surgery $S^3_1(K_n)$ along $K_n$ for $n > 4$ bounds a smooth rational ball. For example, by the $0$-bipolarity of $K_n$ for $n > 0$, the Ozsv\'ath-Szab\'o correction term $d$-invariant \cite{OS03b} vanishes. It is well known that $S^3_1(K_n)$ is the Brieskorn homology sphere $\Sigma(2, 3, 6n+1)$ for positive $n$. We close the introduction with the following classical question:

\begin{question}\cite[Problem 4.2]{Kir97}\cite[Problem Z]{Sav24}
  Which $\Sigma(2, 3, 6n+1)$ bounds a smooth rational ball? 
\end{question}

\subsection*{Organization} In Section 2, we provide a rational slice obstruction using infinitely many branched covers, as stated in Theorem \ref{thm-B}. In Section 3, we obtain a definite filling of the $p^k$-fold branched cover of $K_{m,n}$ for each prime $p$ and prove Theorem \ref{thm-A} by analyzing their intersection forms. In Section 4, we study the rational concordance in the topological category and provide appropriate obstructions for twist knots including Theorem \ref{thm-D} and prove Theorem \ref{thm-C}.
\subsection*{Acknowledgement}
The author appreciates his advisor, JungHwan Park, for his constant support and valuable guidence. The author would also like to thank Marco Golla for suggesting a nice idea that helped merge Lemma \ref{lem-main} with what was previously Lemma 3.6, thereby shortening the proof. Taehee Kim provided important comments as well. The author is also grateful to his colleagues Dongjun Lee, Seungyeol Park, and O{\u{g}}uz \c{S}avk for reading the draft and offering helpful feedback. This work is partially supported by Samsung Science and Technology Foundation (SSTF-BA2102-02) and the NRF grant RS-2025-00542968.

\section{Rational slice obstruction via branched covers}\label{sec-2}
In this section, we briefly recall a slice obstruction via a branched cover and provide a rational slice obstruction Theorem \ref{thm-B} in a similar way.
\begin{theorem}\cite{CG78}\label{thm-CG-pk}
  If a knot $K$ is slice, then, for any prime $p$, the $p^k$-fold cyclic branched cover $\Sigma_{p^k}(K)$ bounds a smooth rational ball.
\end{theorem}
\noindent Such a rational ball filling is obtained by taking the $p^k$-fold branched cover of $B^4$ branched over a slice disk of $K$. Similarly, if $K$ bounds a smooth disk $\Delta$ in a rational ball $V$, one can prove Theorem \ref{thm-B} as a rational slice analogue of Theorem \ref{thm-CG-pk} by taking the prime power-fold branched cover of $V$ branched over the slice disk $\Delta$ for sufficiently large primes.

\begin{proof}[Proof of Theorem \ref{thm-B}]
  Suppose $K$ bounds a smooth slice disk $\Delta$ in a rational ball $V^4$. Note that $V$ is a $\Z_p$-homology ball for any prime $p$ such that $p$ does not divide $|H_1(V;\Z)|$. Let $p$ be any such a prime. Take the $p^k$-fold branched cover $V_{p^k}$ of $V$ branched over $\Delta$. Then $V_{p^k}$ is smooth and bounded by $\Sigma_{p^k}(K)$. It suffices to show that $V_{p^k}$ is a $\Z_p$-homology ball.

  Let $\tilde{V}$ be the infinite cyclic cover of $V-\Delta$. The infinite cyclic group acts on $\tilde{V}$ as a deck transformation. Let $t$ be a generator of the deck transformation group. Then, by \cite{Mil68}, there exists an exact sequence:
  $$\cdots \rightarrow H_i(\tilde{V};\Z_p) \xrightarrow{1-t^{p^k}} H_i(\tilde{V}; \Z_p) \xrightarrow{\hphantom{1-t^p}} H_i(V_{p^k};\Z_p)\rightarrow \cdots.$$
  Similarly, we also have another exact sequence:
  $$\cdots \rightarrow H_i(\tilde{V};\Z_p) \xrightarrow{1-t} H_i(\tilde{V}; \Z_p) \xrightarrow{\hphantom{1-t}} H_i(V;\Z_p)\rightarrow \cdots.$$
  Since $H_i(V;\Z_p) \cong H_i(B^4;\Z_p)$, the map $1-t$ is an isomorphism on $H_i(\tilde{V};\Z_p)$. Then, in $\Z_p$ field, $1-t^{p^k} = (1-t)^{p^k}$ is also an isomorphism on $H_i(\tilde{V};\Z_p)$. Thus, we conclude that $H_i(V_{p^k};\Z_p) \cong H_i(B^4;\Z_p)$.
\end{proof}

\section{Proof of Theorem \ref{thm-A}}\label{sec-3}

In this section, we prove Theorem \ref{thm-A} by applying Theorem \ref{thm-B} proved in the previous section. Thus, it is basically to obstruct a rational homology sphere, which is obtained from the branched cover of a knot, from bounding a smooth rational ball. If a rational homology sphere $Y$ bounds a smooth definite 4-manifold $W$, then there is a chance to use Donaldson's diagonalization theorem \cite{Don87} to obstruct $Y$ bounding a smooth rational ball.

\begin{theorem}\cite{Don87}\label{thm-Donaldson}
  Suppose that a rational homology sphere $Y$ bounds a smooth negative definite 4-manifold $W$. If $Y$ bounds a smooth rational ball, then the intersection form $Q_W$ embeds in $\langle -1 \rangle ^{b_2(W)}$.
\end{theorem}

Recall that the double twist knots $K_{m,n}$ with $mn < 0$ are $0$-bipolar. We use the $0$-bipolarity to obtain a negative definite filling of branched covers of such $K_{m,n}$.

The outline of the proof of Theorem \ref{thm-A} is as follows. First, we take the prime power-fold cyclic branched cover $Y$ of $K_{m,n}$ and find its smooth negative definite filling $W$. Second, we analyze how the intersection form $Q_W$ embeds in $\langle -1 \rangle ^{b_2(W)}$ when $n = -m$ and $n = -m + 1$. Third, we obstruct the lattice embedding of the $N$-fold connected sum $\oplus_N Q_W$ when $n \ge -m + 2$. Note that the second step is to help the reader follow the third step easily, rather than being a necessary part of the proof.

\begin{lemma} \label{lem-Wmn}
  Let $p > 2$ be a prime power, $n > 0$ and $m < 0$. The $p$-fold branched cover $\Sigma_p(K_{m,n})$ of $K_{m,n}$ bounds a smooth negative definite 4-manifold $W_p(m, n)$ described in Figure \ref{fig-pkfold}.

\end{lemma}

\begin{proof}

\begin{figure}[t]
  \includesvg{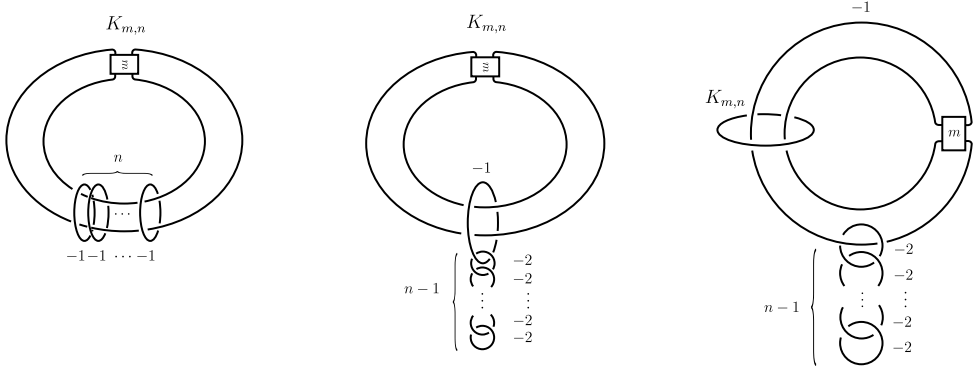}
\caption{$K_{m,n}$ for $n > 0$ bounds a smooth nullhomologous disk in $n\overline{\mathbb{CP}}^2$.}
\label{fig-negdfn}
\end{figure}

Blow up all $n$ positive twists by $(-1)$-framed unknots as illustrated in Figure \ref{fig-negdfn}. Then $K_{m,n}$ bounds a smooth nullhomologous disk $\Delta$ in punctured $n\overline{\mathbb{CP}}^2$, which is negative definite. Let $W_p(m, n)$ be the $p$-fold cyclic branched cover of the punctured $n\overline{\mathbb{CP}}^2$ branched over the disk $\Delta$. It is, of course, bounded by the $p$-fold cyclic branched cover $\Sigma_p(K_{m,n})$ of $S^3$ branched over $K_{m,n}$. By simple Kirby moves and isotopy described in Figure \ref{fig-negdfn}, we obtain the rightmost diagram. From that diagram, we take the $p$-fold branched cover as a 2-handlebody whose Kirby diagram is described in Figure \ref{fig-pkfold}. Since our manifold involves cyclically linked 2-handles, we explain how to get the surgery coefficients and linking numbers with correct signs in detail.

Let $w$ be the writhe of the $(-1)$-framed unknot in the rightmost diagram in Figure \ref{fig-negdfn}. Note that each lift of the $(-1)$-framed unknot is $(-w-1)$-framed. Because each negative twist of two oppositely oriented strands contributes $+2$ to the writhe and we have $|m|$ many negative twists, the framing coefficient of each lift of $(-1)$-framed unknot is $-2|m| - 1 = 2m - 1$. Since two $|m|$-linked strands are directed in opposite directions, the linking number between two linked $(2m-1)$-framed unknots is positive, and hence equal to $|m| = -m$. Thus, we get the 2-handlebody in Figure \ref{fig-pkfold}.

\begin{figure}[t]
  \includesvg[scale=0.9]{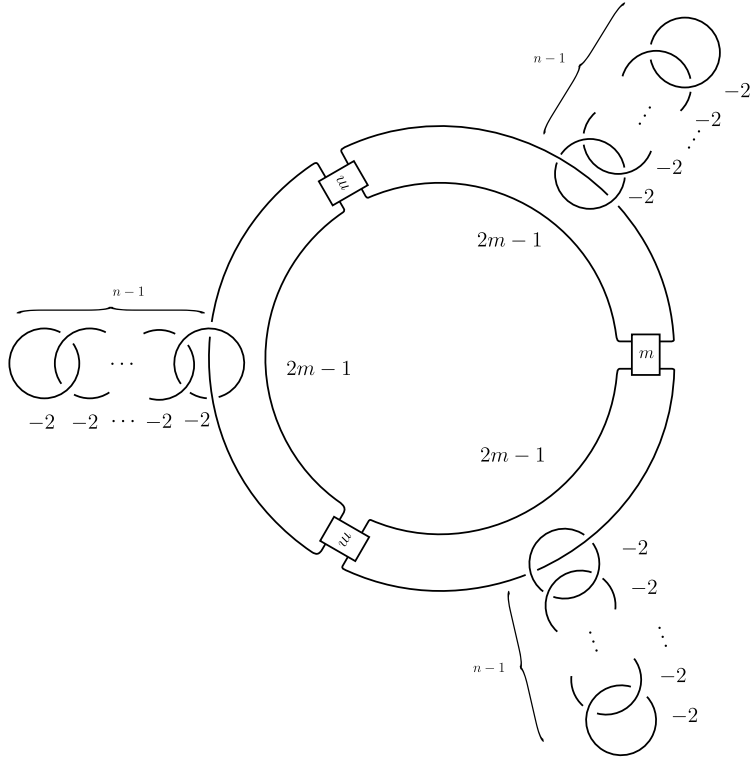}
  \caption{A negative definite filling $W_p(m,n)$ of $\Sigma_p(K_{m,n})$ for $m < 0$, $n > 0$, and $p > 2$.}
\label{fig-pkfold}
\end{figure}

Since $\Delta$ is nullhomologous, by Atiyah-Singer G-signature theorem \cite{AS68}, it can be checked that $W_p(m,n)$ is negative definite. For example, see \cite[Corollary 4.2]{CHH13}. For completeness, we provide a brief proof from \cite{CHH13}.

Take the $p$-fold branched cover $V$ of $B^4$ branched over a Seifert surface $F$ of $K_{m,n}$ and attach $-V$ to $W_p(m, n)$ along $\Sigma_p(K_{m,n}) = \bdry W_p(m,n)$. We obtain a closed smooth 4-manifold $X$ with $\Z_p$-action generated by, say $\tau$. Our $X$ is the $p$-fold branched cover of $n \overline{\mathbb{CP}}^{\smash{2}}$ branched over the closed smooth surface $\Sigma = \Delta \cup F$. Since $\tau^p = id$, the eigenvalues of $\tau$ on $H_2(X;\CC)$ are $\exp(2\pi i k/p)$ for $k = 0, \cdots, p-1$. Let $\sigma(-;k)$ be the signature of the intersection form on the $\exp(2\pi i k/p)$-eigenspace of $H_2(-;\CC)$ with respect to $\tau$. Then,
$$ \sigma(X;k) = \sigma(W_p(m,n);k) - \sigma(V;k).$$
Casson and Gordon \cite{CG78} proved the following identity by applying G-signature theorem: $$\sigma(X;k) = \sigma(n\mathbb{\overline{CP}}^2) - 2[\Sigma]\cdot[\Sigma]\frac{k(p-k)}{p^2}.$$ Since the branched set $\Sigma$ is nullhomologous in $n\overline{\mathbb{CP}}^2$, we have $\sigma(X;k) = \sigma(n \overline{\mathbb{CP}}^{\smash{2}}) = -n$ for each $k$ and hence, $$\sigma(X) = \sum\limits_{k= 0}^{p-1}\sigma(X;k) = -pn.$$ 

To prove that $W_p(m,n)$ is negative definite, we have to show that $\sigma(W_p(m,n)) = -b_2(W_p(m,n))$. Since $-b_2(W_p(m,n)) = -pn = \sigma(X)$, it is enough to check that $\sigma(V;k)$ vanishes for each $k = 0,\cdots, p - 1$ by the above equality by Casson and Gordon. Recall that $V$ is the $p$-fold branched cover of $B^4$ branched over Seifert surface $F$ of $K_{m,n}$. It is well-known \cite{Vir73} that $$\sigma(V;k) = \sigma_{K_{m,n}}(\exp(2\pi i k/p)),$$ where the right-hand side denotes the Levine-Tristram signature function of $K_{m,n}$. Since $K_{m,n}$ is $0$-bipolar, the signature function vanishes, so the proof is complete.
\end{proof}

\noindent Let $p > 2, m < 0,$ and $n > 0$. The intersection form $Q_p(m, n)$ of $W_p(m, n)$ is the $n$-by-$n$-block matrix: 
  $$Q_p(m, n) = \begin{bmatrix}
    Q_p(m,1) & I_p & O &\cdots &\cdots &O \\
    I_p & -2I_p & I_p & O &\cdots &O\\
    O & I_p & -2I_p & I_p & \cdots & O\\
    \vdots & \ddots &\ddots &\ddots &\ddots & \vdots \\
    O & \cdots & O & I_p & -2I_p & I_p\\
    O & \cdots & \cdots & O & I_p & -2I_p
  \end{bmatrix},$$ where each block is a $p$-by-$p$ matrix, $I_p$ is the rank $p$ identity matrix, and $Q_p(m, 1)$ is the following $p$-by-$p$ matrix:
  $$Q_p(m, 1) = \begin{bmatrix}
    2m -1 & -m & 0 & \cdots & 0 & -m\\
    -m & 2m-1 & -m & 0 & \cdots & 0\\
    0 & -m & 2m-1 & -m & \cdots & 0\\
    \vdots & \ddots &\ddots &\ddots & \ddots & \vdots \\
    0 & 0 & \cdots & -m & 2m-1 & -m\\
    -m & 0 & \cdots & 0 & -m & 2m-1\\
  \end{bmatrix}.$$\\

  Recall that when $n= -m + 1$, $K_{m,n}$ is slice and when $n = -m$, $K_{m,n}$ is rationally slice in a $\Z_p$-homology ball for any odd prime $p$. Thus, in those cases, $Q_p(m,n)$ must be embedded into $\langle -1 \rangle ^{np}$. As a warm-up for the obstruction part, we give explicit examples of their embeddings.

  \begin{example}\label{ex-1}
    Let $a_1, \cdots, a_p, b_1, \cdots, b_p, c_1, \cdots, c_{(n-2)p}$ be the standard basis of $\langle -1 \rangle ^{np}$. Suppose $n = -m + 1$. We set $v_i, w_i, x_i,$ and $x_{pk+i}$ up for each $i \in \{1, \cdots, p\} = \Z_p$ and $k = 1,\cdots, n-3$ as:
    \begin{align*}
      v_i &= b_i + c_i + c_{p+i}+ \cdots + c_{(n-3)p + i} - (a_{i + 1} + b_{i+1} + c_{i+1} + c_{p+i+1} +  \cdots + c_{(n-3)p + i + 1}),\\
      w_i &= a_i - b_i, \quad x_i = b_i - c_i, \quad x_{pk+i} = c_{p(k-1)+i} - c_{pk + i}.
    \end{align*} Then, one can check that
    \begin{align*}
      v_i &\cdot v_i = -(n-1) - n = 2m - 1,\\
      v_i &\cdot v_{i \pm 1} = n-1 = -m.
    \end{align*} Moreover, $(\Z^{np}, Q_p(m,n))$ is spanned by $v_1, \cdots, v_p, w_1, \cdots, w_p, x_1, \cdots, x_{(n-2)p}$. Hence, $Q_p(m, -m + 1)$ embeds in $\langle -1 \rangle ^{np}$.

    Similarly, consider the case when $n = -m$. In this case, we restrict our $p$ as a odd prime power $2q + 1$. Set $w_i, x_i$, and $x_{pk+i}$ as the same as before and let $v_i$ be as follows:
\begin{align*}
  v_i = -a_i &+ (a_{i+q} + b_{i+q} + c_{i + q} + c_{p + i + q} + \cdots + c_{p(n-3) + i + q})\\ &- (a_{i+q+1} + b_{i + q + 1} + c_{i + q+1} + c_{p + i + q+1} + \cdots + c_{p(n-3) + i + q+1}). 
\end{align*}
Then $v_1, \cdots, v_p, w_1, \cdots, w_p, x_1,\cdots, x_{(n-2)p}$ span $Q_p(m, -m)$ so that we have an embedding of $Q_p(m, -m)$ in $\langle -1 \rangle ^{np}$.
  \end{example}
\begin{lemma}\label{lem-main}
  Let $p> 2$. If $n \ge -m+2$, then $\oplus_N Q_p(m, n)$ does not embed in $\langle -1 \rangle ^{Nnp}$ for any $N \ge 1$.
\end{lemma}
\begin{proof}
  For each $s=1, \cdots, N$, let $v_1^s, \cdots, v_p^s, w_1^s, \cdots, w_p^s, x_1^s, \cdots, x_{(n-2)p}^s$ be a basis of the $s$-th summand of $(\Z^{Nnp}, \oplus_N Q_p(m, n))$ in the order from left to right. We denote $\oplus_N Q_p(m, n)(a, b)$ by $a\cdot b$. Then, 

  \begin{align*}
    v_i^s \cdot v_j^t &= \begin{cases}
      2m-1 & \text{ if }(s, i) = (t, j),\\
      -m & \text{ if } s = t \text{ and }i - j \equiv \pm 1 \text{ in }\Z_p,\\
    0 & \text{ otherwise.}\end{cases} \qquad 
    &w_i^s \cdot w_j^t &= \begin{cases}
      -2 & \text{ if }(s, i) = (t, j),\\
      0 & \text{ otherwise.}\end{cases}\\ 
    v_i^s \cdot w_j^t &= \begin{cases}
      1 & \text{ if }(s, i) = (t, j),\\
      0 & \text{ otherwise.}\end{cases} \qquad 
    &w_i^s \cdot x_j^t &= \begin{cases}
      1 & \text{ if }(s, i) = (t, j),\\
      0 & \text{ otherwise.}\end{cases}\\ 
      v_i^s \cdot x_j^t &= 0. \qquad
    &x_i^s \cdot x_j^t &= \begin{cases}
      -2 & \text{ if }(s, i) = (t, j),\\
      1 & \text{ if } s = t \text{ and }i - j = \pm p,\\
      0 & \text{ otherwise.}\end{cases}
  \end{align*}
  Suppose $\oplus_N Q_p(m, n)$ embeds in $\langle -1 \rangle ^{Nnp}$. By abuse of notation, we use $v_i^s, w_i^s, x_i^s$ to denote their images in $(\Z^{Nnp}, -I_{Nnp})$. Then each of $v_i^s$, $w_i^s$, and $x_i^s$ must be a linear combination of the standard basis $\{e_1^t, \cdots, e_{np}^t\}_{t=1}^N$ of $(\Z^{Nnp}, -I_{Nnp})$. For simplicity, we also write $-I_{Nnp}(a, b)$ as $a\cdot b$. We say $v$ \textit{meets} $z$ if $v\cdot z\neq 0$ and $v$ \textit{is supported by} standard basis elements $z_1,\cdots, z_k$ when $v\cdot z = 0$ for a standard basis $z$ if and only if $z\not\in \{\pm z_1, \cdots, \pm z_k\}$. Let $V_i^s, W_i^s$ and $X_i^s$ be the set consisting of $e_k^t$ such that $v_i^s\cdot e_k^t\neq 0$, respectively for $w_i^s, x_i^s$.

  Before we proceed with the proof, we provide a rough guideline. First, we claim that $w_i^s$ and $x_i^s$ should be in the form given in the above example. Second, we find an expression for $v_i^s$ using $v_i^s\cdot w_i^t$ and $v_i^s \cdot x_i^t$. Using $v_i^s\cdot v_i^s = 2m - 1$ with the assumption $n\ge -m + 2$, we are left with only three cases: $n = -2m \pm 2$ or $n = -2m$. For each case, we obstruct lattice embeddings by using $v_i^s\cdot v_j^s = -m$ for $j= i \pm 1$.\\

  \noindent\textbf{Claim 1.} $|W_i^s\cap W_j^t|= 0$ when $(s,i)\neq (t,j)$.\\

  Since $w_i^s\cdot w_i^s = -2$, we may assume that $w_i^s = a_i^s - b_i^s$ for some $a_i^s, b_i^s\in\{\pm e_1^t,\cdots, \pm e_{np}^t\}_{t=1}^N$ such that $a_i^s\cdot b_i^s = 0$. If $|W_i^s\cap W_j^t|\neq 0$, then $W_i^s = W_j^t$ since $w_i^s\cdot w_j^t = 0$. Then we may assume that $w_j^t = \pm (a_i^s + b_i^s)$. Then $x_i^s\cdot w_j^t = x_i^s \cdot \pm (w_i^s + 2b_i^s) = \pm (1 +2x_i^s\cdot b_i^s) \neq 0$, which contradicts to $x_i^s\cdot w_j^t = 0$. Thus, \textbf{Claim 1} is done.\\

  \noindent Without any loss of generality, we still let $$w_i^s = a_i^s - b_i^s,$$ where $a_i^s, b_i^s\in \{\pm e_1^t, \cdots, \pm e_{np}^t\}_{t=1}^N$ and each one is orthogonal to the others.\\

  \noindent\textbf{Claim 2.} $|X_i^s\cap W_j^t|= 0$ when $(s,i)\neq (t,j)$.\\

  Otherwise, $X_i^s = W_j^t$ since $x_i^s\cdot w_j^t = 0$. It is clear from $x_i^s \cdot w_i^s = 1$ that $|X_i^s \cap W_i^s| = 1$, we have $|W_i^s \cap W_j^t| = 1$ for $(s,i)\neq (t,j)$, which contradicts to \textbf{Claim 1.}\\


  \noindent Since $|X_i^s\cap W_i^s| = 1$ and $|X_i^s\cap W_j^t| = 0$ for $i= 1,\cdots, p$ and any $(t, j) \neq (s, i)$, $x_i^s$ meets precisely one of $a_i^s$ and $b_i^s$. If $x_i^s$ meets $a_i^s$, then by setting new $(a_i^s, b_i^s)$ as $(-b_i^s, -a_i^s)$, we may assume that $$x_i^s = b_i^s - c_i^s$$ for some $c_i^s \in \{\pm e_1^t,\cdots, \pm e_{np}^t\}_{t=1}^N$ such that each of $a_i^s$ and $b_i^s$ is orthogonal to any $c_j^t$. Also note that $c_i^s \cdot c_j^t = 0$ for $(s,i)\neq (t,j)$ since $x_i^s\cdot x_j^t = 0$. Since $x_{p+i}^s\cdot x_i^s = 1$, $x_{p + i}^s$ meets exactly one of $b_i^s$ and $c_i^s$. If $x_{p+i}^s$ met $b_i^s$, then $x_{p+i}^s = -b_i^s-a_i^s$ since $x_{p+i}^s\cdot w_i^s = 0$. However, this contradicts to $x_{p+i}^s\cdot v_i^s = 0$ since $x_{p+i}^s\cdot v_i^s = (w_i^s-2a_i^s)\cdot v_i^s = 1 - 2a_i^s\cdot v_i^s\neq 0$. By a similar argument, we may assume that for $k = 1,\cdots, n-3$, $$x_{pk + i}^s = c_{p(k-1)+i}^s - c_{pk + i}^s,$$ where $\{a_i^s, b_i^s, c_{pk+i}^s\}_{s=1}^N$ spans $(\Z^{Nnp}, -I_{Nnp})$ and each of them is orthogonal to the others.\\

  \noindent\textbf{Claim 3.} $v_i^s \cdot b_j^t = v_i^s \cdot c_j^t = v_i^s \cdot c_{p+j}^t = \cdots = v_i^s \cdot c_{(n-3)p + j}^t = -l_{i,j}^t$ and $v_i^s \cdot a_j^t = -l_{i,j}^t + \delta_{(s,i)}^{(t,j)}$ where $|l_{i,j}^t| \le 1$.\\

  Let $v_i^s \cdot b_j^t = -l_{i,j}^t$. Then it follows from $v_i^s \cdot x_j^t = v_i^s\cdot (b_j^t-c_j^t) = 0$ that $v_i^s \cdot c_j^t = -l_{i,j}^t$. Similarly, since $v_i^s \cdot x_{pk + j}^t = 0$, we have $v_i^s \cdot c_{pk + j}^t = l_{i,j}^t$. The first part is done by induction on $k$. Since $v_i^s\cdot w_j^t = \delta_{(s,i)}^{(t,j)} = v_i^s\cdot (a_j^t-b_j^t) = v_i^s \cdot a_j^t + l_{i,j}^t$, the second part is done. Suppose some $|l_{i,j}^t| > 1$. Then $v_i^s\cdot v_i^s \le -(l_{i,j}^t)^2 (n-1) \le -4(n-1) \le -4(-m +1) < 2m - 1$, which completes the proof of \textbf{Claim 3}.\\

  \noindent We regard the index $i$ as an element in $\{1, \cdots, p\}\in \Z_p$. From \textbf{Claim 3}, we can let $v_i^s$ be of the following form: $$v_i^s = -a_i^s + \sum\limits_{t=1}^N\sum\limits_{k = 1}^p l_{i,k}^t (a_k^t + b_k^t + c_k^t + c_{p+k}^t + \cdots + c_{(n-3)p + k}^t),$$
  where each $l_{i,k}^t$ has absolute value at most $1$. Then we have the following:

  $$v_i^s \cdot v_j^s = a_i^s\cdot a_j^s + l_{j,i}^s + l_{i,j}^s -n\sum\limits_{t=1}^N \sum\limits_{k=1}^p l_{i,k}^t l_{j,k}^t.$$

  \noindent When $j = i$,
  \begin{equation}\label{eqn-1}
    2 l_{i,i}^s -n\sum\limits_{t=1}^N\sum\limits_{k=1}^p (l_{i,k}^t)^2 = 2m.
\end{equation}

  \noindent When $j = i \pm 1$,
  \begin{equation}\label{eqn-2}
    -l_{j,i}^s - l_{i,j}^s + n\sum\limits_{t=1}^N\sum\limits_{k=1}^p l_{i,k}^t l_{j,k}^t = m.
  \end{equation}

  Applying the assumption that $n \ge -m + 2$ to (\ref{eqn-1}), we have the inequality:
  $$\sum\limits_{s=1}^N\sum\limits_{k=1}^p(l_{i,k}^t)^2 \le \frac{2m - 2l_{i,i}^s}{m-2} = 2 + \frac{4 - 2l_{i,i}^s}{m-2} < 2.$$
  Note that the summation $S$ in the left-hand side of the inequality is a nonnegative integer less than $2$. If $S$ were $0$, then $l_{i,k}^t=0$ for all $k = 1, \cdots, p$ and $t=1,\cdots N$ so that $m = 0$ from (\ref{eqn-1}). Thus, $S$ must be $1$ and
  $$ -n + 2l_{i,i}^s = 2m.$$
  Before we start to deal with each case, notice that the assumption $n\ge -m+2$ was crucial in deriving $S = 1$. If we assumed $n \ge -m + c$ for some $c =0$ or $1$, then the strict inequality between $S$ and $2$ does not hold anymore and it might happen that $S = 2$ like the cases in Example \ref{ex-1}. \\

  \noindent \textbf{Case 1.} $l_{i,i}^s = -1$ and $n = -2m - 2$.\\

  In this case, $l_{i,k}^t = 0$ whenever $(t, k)\neq (s,i)$. Also note that $n = -2m -2$ and $m \le -4$ by the assumption $n \ge -m + 2$. Since $l_{i,j}^s = 0$ when $j = i\pm 1$, we obtain the following from (\ref{eqn-2}):
  $$m = -l_{j,i}^s - nl_{j,i}^s.$$

  \noindent In any case where $l_{j,i}^s\in \{-1, 0,1\}$, the condition $m \le -4$ is violated.\\

  \noindent \textbf{Case 2.} $l_{i,i}^s = 1$ and $n = -2m + 2$.\\

  In this case, we similarly obtain the following from (\ref{eqn-2}):
  $$m = -l_{j,i}^s + nl_{j,i}^s.$$

  \noindent In any case where $l_{j,i}^s\in \{-1, 0, 1\}$, the condition $m \le 0$ is violated.\\

  \noindent \textbf{Case 3.} $l_{i,i}^s = 0$ and $n = -2m$.\\

  In this case, there exists a unique $(t,k)\neq (s,i)$ such that $l_{i,k}^t = \pm 1$. Also note that $n = -2m$ and $m \le -2$ by the assumption $n \ge -m + 2$. Then when $j = i \pm 1$, it follows from (\ref{eqn-2}) and $n = -2m$ that:
  $$m = \frac{-l_{j,i}^s - l_{i,j}^s}{2l_{i,k}^t l_{j,k}^t + 1}.$$
  Since $i+1 \not\equiv i -1$ mod $p$ for $p > 2$, we can always choose $j\in \{i+1, i-1\}$ such that $k \neq j$. Then $l_{i,j}^s = 0$ so that:
  $$m = \frac{-l_{j,i}^s}{2l_{i,k}^t l_{j,k}^t + 1}.$$ Since $l_{i,k}^t l_{j,k}^t$ and $l_{j,i}^s$ are always either $-1, 0$, or $1$, it contradicts to $m\le -2$ for any cases. The assumption has also been necessary here because $K_{-m, 2m}$ is slice when $m=1$.\\

  Therefore, we can conclude that $(\Z^{Nnp}, \oplus_N Q_p(m, n))$ does not embed in $(\Z^{Nnp}, -I_{Nnp})$ whenever $n \ge -m + 2$.
  \end{proof}

  \begin{remark}
    One might guess that the $0$-negativity of $K_{m,n}$ would be enough since we have used only negative definite filling but not a positive one. However, we have made use of both $0$-negativity and $0$-positivity at two points. First, when we proved our branched cover of $n\overline{\mathbb{CP}}^{\smash{2}}-\Delta$ is negative definite in Lemma \ref{lem-Wmn}, recall that we used the fact that $\sigma_{K_{m,n}}(\omega) = 0$. In general, $K$ being $0$-negative is not sufficient to guarantee that $\sigma_K(\omega)$ vanishes.

    Second, our lattice embedding obstruction on the negative definite filling does not always work unless $n \ge -m + 2$. For example, consider the knot $K_{m, 1}$ with $m < 0$. One can find that the intersection form $Q_p(m, 1)$ of $W_p(m, 1)$ embeds in the standard negative definite $(\Z^{p}, -I_{p})$ when $-m = k^2$. Instead, we take its mirror $K_{-1, -m}$. Since this knot is $0$-bipolar, the $p$-fold branched cover of $K_{-1, -m}$ still has a negative definite filling and we obstruct the lattice embedding of $Q_p(-1, -m)$. 
  \end{remark}

  \begin{remark}
    Marco Golla has informed us that \textbf{Claims 1} and \textbf{2} follow from a well-known fact about the uniqueness of the lattice embeddings of so-called \textit{$2$-legs} or \textit{chain of twos}. For example, see \cite{AGLL20, AMP22}. Nevertheless, we include the proof for the sake of completeness of this article.
  \end{remark}

\begin{proof}[Proof of Theorem \ref{thm-A}]
  It is known that $K_{m, -m\pm 1}$ is slice by \cite{Sie75} and $K_{m, -m}$ is rationally slice \cite{Cha07}. Thus, (b) implies (a), and it is trivial that (a) implies (c). We prove that (c) implies (b). We proceed by contrapositive, so we show that the double twist knot $K_{m,n}$ have infinite order in the rational knot concordance group $\C_\Q$ whenever $mn\neq 0, n \neq -m \pm 1$, and $n \neq -m$.

  Note that the signature $\sigma(K_{m,n}) \neq 0$ when $mn > 0$ and $\sigma$ is an integer-valued homomorphism from $\C_\Q$ \cite{CO93, CK02}, $K_{m,n}$ has infinite order in $\C_\Q$. We assume $mn < 0$. By the symmetry $K_{m,n} = K_{n,m}$ and taking mirror if it is needed, we may assume that $n \ge -m$. By the hypothesis that $n \neq -m, -m\pm 1$, it is enough to prove that $K_{m,n}$ for $n \ge -m + 2$ has infinite order in $\C_\Q$ when $n\ge -m + 2$.

  Suppose $\#_NK_{m, n}$ with $n \ge -m + 2$ is rationally slice. Then the $p$-fold cyclic branched cover $\#_N\Sigma_p(K_{m,n})$ bounds a smooth rational ball $V$ for sufficiently large prime $p$ by Theorem \ref{thm-B}. Take the negative definite filling $\natural_N W_p(m,n)$ of $\#_N \Sigma_p(K_{m, n})$ as given in Lemma \ref{lem-Wmn}. Then Theorem \ref{thm-Donaldson} implies that the intersection form $\oplus_N Q_p(m,n)$ of $\natural_N W_p(m,n)$ embeds in $\langle -1 \rangle ^{Nnp}$. However, by Lemma \ref{lem-main}, it is impossible for any prime $p > 2$.
\end{proof}

\section{Rational concordance of twist knots in the topological category}\label{sec-4}
Recall that Donaldson's theorem used in the previous section works in the smooth category. However, the topological concordance is much different from the smooth concordance. In this section, we prove Theorem \ref{thm-C}, which recovers a partial result for the twist knots $K_n = K_{-1, n}$ of Theorem \ref{thm-A} in the topological category. We first pick out the algebraically rationally slice twist knots. Then we obstruct their topological rational sliceness by employing von Neumann $\rho$-invariant. Througout this section, all discussion on the concordance is carried out in the topological category. For example, a rationally slice knot means a knot bounding a locally flat disk in a topological rational ball.

\subsection{Rational algebraic concordance}
Levine \cite{Lev69} defined a quotient map $\phi$ of the knot concordance group $\C$ by associating the S-equivalence class of a Seifert form. The quotient image $\AC$ is called the \textit{algebraic concordance group}. He also proved that $\AC\cong \Z^\infty\oplus \Z_2^\infty\oplus \Z_4^\infty$. Note that the algebraic concordance class is a topological knot concordance invariant. One simple way to obstruct a knot being slice is to check if its algebraic concordance class is nontrivial. For example, we can use the Alexander polynomial.

\begin{theorem}[Fox-Milnor condition]\label{thm-FM}
  If a knot is algebraically slice, then its Alexander polynomial $\Delta(t)$ satisfies:
  \begin{equation*}
    \Delta(t) = f(t)f(t^{-1}).
  \end{equation*}
\end{theorem}
\noindent Note that the above equality means that both sides are the same up to multiplication by a unit in $\Zt$. The converse, in general, does not hold. For example, consider a knot $K$ whose algebraic concordance order is $4$. Then the connected sum $K\#K$ satisfies the splitting in Theorem \ref{thm-FM}, but its order in $\AC$ is $2$. Nevertheless, there is a partial converse:
\begin{theorem}\cite{Cha07}\label{thm-FM-conv}
  Let $K$ be a knot, and let $\Delta(t)$ be its Alexander polynomial. If each irreducible factor of $\Delta(t)$ is not symmetric up to a unit in $\Zt$, then $K$ is algebraically slice. In particular, if
  \begin{equation*}
    \Delta(t) = f(t)f(t^{-1})
  \end{equation*}
  for some irreducible $f(t)$ such that $f(t) \neq f(t^{-1})$ up to a unit, then $K$ is algebraically slice.
\end{theorem}
\begin{proof}
  For any reciprocal number $z$, its irreducible polynomial $\lambda(t)$ is symmetric so that it does not divide $\Delta(t)$ by assumption. Then by \cite[Proposition 3.6 (1)]{Cha07}, the $z$-primary part of $\phi(K)$ is trivial. Since this holds for every reciprocal number $z$, $\phi(K)$ is trivial in $\AC$.
\end{proof}

As an analogue for the rational knot concordance group $\C_\Q$, Cha \cite{Cha07} defined a quotient map $\phi_\Q: \C_\Q\rightarrow \AC_\Q$ such that below diagram commutes:

\begin{center}
\begin{tikzcd}
&\C \ar[twoheadrightarrow]{r}{\psi}\ar[twoheadrightarrow]{d}{\phi} &\C_\Q \ar[twoheadrightarrow]{d}{\phi_\Q}\\
&\AC \ar[twoheadrightarrow]{r}{\overline{\psi}} &\AC_\Q
\end{tikzcd}
\end{center}
where the quotient image of $\C_\Q$ is called the \textit{algebraic rational concordance group} $\AC_\Q$. Cha also used an analogous version of Fox-Milnor condition implicitly in \cite{Cha07}. See also \cite[Proposition 4.5]{CFHH13} for an explic statement.

\begin{theorem}[generalized Fox-Milnor condition]\cite{Cha07}\label{thm-FM-Q}
  If a knot is algebraically rationally slice, then its Alexander polynomial $\Delta(t)$ satisfies for some positive integer $c$:
  \begin{equation*}
    \Delta(t^c) = f(t)f(t^{-1}).
  \end{equation*}
\end{theorem}

\noindent The equality holds up to a unit in $\Zt$. There is also a partial converse as a rational analogue of Theorem \ref{thm-FM-conv}, proved in \cite{Cha07}.

\begin{theorem}\cite[Below of Example 3.17]{Cha07}\label{thm-FM-Q-conv}
  Let $K$ be a knot, and let $\Delta(t)$ be its Alexander polynomial. If each irreducible factor of $\Delta(t^c)$ for some positive integer $c$ is not symmetric up to unit in $\Zt$, then $K$ is algebraically rationally slice. In particular, if
  \begin{equation*}
    \Delta(t^c) = f(t)f(t^{-1})
  \end{equation*}
  for some irreducible $f(t)$ such that $f(t) \neq f(t^{-1})$ up to a unit, then $K$ is algebraically rationally slice.
  
\end{theorem}
\begin{proof}
  Let $\A\in\AC_\Q$ be the algebraic rational concordance class of $K$. By assumption, the generalized Seifert form $A$ of complexity $c$ is trivial by Theorem \ref{thm-FM-conv}. Since $\A$ is defined as the image of $A$ under the homomorphism $\phi_c$ defined in \cite[Definition 2.22]{Cha07}, the proof is done.
\end{proof}

A polynomial $p(t)$ is called \textit{strongly irreducible} if $p(t^c)$ is irreducible for all positive integer $c$. It was previously shown that the Alexander polynomial of the twist knot $K_n$ for not perfect power $n$ is strongly irreducible \cite{BD12}. It is, however, difficult to prove that a given polynomial is strongly irreducible in general. Instead of that, we can use Theorems \ref{thm-FM-Q} and \ref{thm-FM-Q-conv} directly to prove Theorem \ref{thm-ACQ}:
\begin{theorem}\label{thm-ACQ}
  $K_n$ is algebraically rationally slice if and only if $n=k(k-1)$ or $n=k^2$ for some $k$.
\end{theorem}

\begin{proof}
Let $\Delta_n(t)$ be the Alexander polynomial of $K_n$. Then 
\begin{equation*}
  \Delta_n(t) = nt^2 - (2n + 1)t + n.
\end{equation*} If $n = k(k-1)$, then it was already known that $K_n$ is algebraically slice \cite{Lev69}. By using the fact that  
\begin{equation*}
  \Delta_n(t) = (kt - (k-1))((k-1)t - k)
\end{equation*} and $kt - (k-1)$ is not symmetric, it also follows from Theorem \ref{thm-FM-conv} that $K_n$ is algebraically slice.

\noindent For the case when $n = k^2$, it was previously shown that $K_n$ is algebraically rationally slice in \cite[Remark 4.6]{BD12} by computing the complete invariants $s, e$, and $d$ for $\AC_\Q$ defined by Cha \cite{Cha07}. For the reader's convenience, we reprove it without computing $s, e$, and $d$. Since 
  \begin{equation*}
    \Delta_n(t^2) = (kt^2-t - k)(kt^2+t-k)
  \end{equation*}
  and $kt^2\pm t-k$ is obviously irreducible and not symmetric up to unit, by Theorem \ref{thm-FM-Q-conv}, $K_n$ is algebraically rationally slice. See also \cite[Example 3.17]{Cha07}.

  Now we prove the converse. Suppose $K_n$ is algebraically rationally slice. Then, by Theorem \ref{thm-FM-Q}, $\Delta_n(t^c) = uf(t)f(t^{-1})$ for some unit $u\in \Zt$, where $f$ is a degree $k$ polynomial. Without loss of generality, we may assume that $u = \pm t^c$ and $f(t)\in \Z[t]$ with degree $c$. Write $f(t) = a_c t^c + a_{c-1} t^{c-1} + \cdots + a_1 t + a_0$ and let $g(t) = t^c f(t^{-1})$ and $\varepsilon = u / t^c$. Then,
  \begin{align*}
    \Delta_n(t^c) &= nt^{2c} - (2n + 1)t^c + n = \varepsilon f(t)g(t)\\
    &= \varepsilon(a_c t^c + a_{c-1} t^{c-1} + \cdots + a_1 t + a_0)(a_0 t^c + a_1 t^{c-1} + \cdots + a_{c-1} t + a_c)\\
    &= \varepsilon(a_c a_0 t^{2c} + \cdots + (a_c^2 +\cdots + a_0^2)t^c + \cdots + a_0 a_c).
  \end{align*}
  Then $n = \varepsilon a_c a_0$, so the coefficient of $t^c$ term is $\varepsilon\sum\limits_{i=0}^c a_i^2 = -2\varepsilon a_c a_0 - 1$. Then $(a_c + a_0)^2 + \sum\limits_{i=1}^{c-1} a_i^2 = -\varepsilon$. Since the left-hand side is nonnegative, the right-hand side is $+1$. Thus, $|a_c + a_0|\le 1$, which implies that $n=k(k-1)$ or $n=k^2$.
\end{proof}

Thus, to prove Theorem \ref{thm-C}, it suffices to obstruct the algebraically rationally slice twist knots from being actually rationally slice. This will be dealt with in the next subsection.

\subsection{von Neumann $\rho$-invariant with complexity}
In this subsection, we employ the von Neumann $\rho$-invariant with complexity to prove Theorem \ref{thm-C}. As one can see in Theorem \ref{thm-ACQ}, the remaining case splits into two parts:
\begin{itemize}
  \item $K_n$ is trivial in $\AC$,
  \item $K_n$ is nontrivial in $\AC$ but is trivial in $\AC_\Q$.
\end{itemize}

\noindent Since the product of any two nonzero consecutive integers cannot be a square, the first case corresponds exactly to $n = k(k-1)$, and the second case corresponds exactly to $n = k^2$. Thus, it is enough to prove the following two theorems:

\begin{theorem}\label{thm-AC0}
  If $n = k(k-1)$ with $k > 2$, then $K_n$ is not rationally slice.
\end{theorem}

\begin{theorem}\label{thm-AC2}
  If $n = k^2$ with $k > 8$, then $K_n$ is not rationally slice.
\end{theorem}

\begin{proof}[Proof of Theorem \ref{thm-C} modulo Theorems \ref{thm-AC0} and \ref{thm-AC2}]
  By assumption, $n \neq 2^2, 3^2, \cdots, 8^2$. Suppose that $K_n$ is rationally slice. Then it is algebraically rationally slice. By Theorem \ref{thm-ACQ}, $n$ is either $k(k-1)$ or $k^2$. Consider the former case. Then by Theorem \ref{thm-AC0}, $K_n$ is rationally slice only when $n = 0$ or $2$. For the latter case, by Theorem \ref{thm-AC2}, $K_n$ is rationally slice only when $n=1$.
\end{proof}

\noindent In the rest of this section, we will provide appropriate obstructions for each case and use them to prove Theorems \ref{thm-AC0} and \ref{thm-AC2}. We first recall the definition of von Neumann $\rho$-invariant for a closed 3-manifold with $b_1 = 1$ and its basic property. Throughout this section, every manifold is orientable, compact, and connected.

\begin{definition}
  Let $M$ be a closed 3-manifold with $b_1(M)=1$, $\Gamma$ be a discrete group, and $\phi:\pi_1(M)\rightarrow \Gamma$. Suppose that $M$ bounds a 4-manifold $W$ such that the inclusion induces an isomorphism on the first homology with rational coefficient and $\phi$ extends to $\psi:\pi_1(W)\rightarrow \Gamma$ as in the commutative diagram below.

\begin{center}
\begin{tikzcd}
  \pi_1(M)\ar{r}{\phi}\ar{d} &\Gamma\\
  \pi_1(W)\ar[swap]{ur}{\psi}
\end{tikzcd}
\end{center}

Let $\sigma^{(2)}(W, \psi)$ denote the $L^2$-signature of $W$ with respect to $\psi$. Then \textit{the von Neumann $\rho$-invariant for $(M,\rho)$} is defined to be the \textit{signature defect} of $(W, \psi)$, namely: 
\begin{equation*}
    \rho(M, \phi) = \sigma^{(2)}(W, \psi) - \sigma (W).
  \end{equation*}
\end{definition}

\noindent The $L^2$-signature $\sigma^{(2)}(W,\psi)$ is, roughly speaking, the signature of the intersection form of $\psi$-covering of $W$. For a rigorous definition of $L^2$-invariant, see \cite{ChG85, Luc02, COT03, Cha08}. The value $\rho(M,\phi)$ is independent of choices of $W$ and $\psi$ as proved in \cite{CW03, COT03}.

The following lemma, called the \textit{subgroup property}, is a well-known fact for the $\rho$-invariant. For example, see \cite[Proposition 5.13]{COT03}.

\begin{lemma}\label{lem-subgp-property}
  If $j:\Gamma\rightarrow \Lambda$ is injective, then
  \begin{equation*}
    \rho(M, \phi) = \rho (M, j\circ \phi).
  \end{equation*}
\end{lemma}

Recall that the \textit{Alexander module} $\A(K)$ is defined as $H_1(M_K;\Qt)$, which is a torsion $\Qt$-module, and the \textit{Blanchfield form} $\Bl$ is the $\Q(t)/\Qt$-valued linking form on $\A(K)$. A submodule $P$ of $\A(K)$ is said to be \textit{isotropic} if $P\subset P^\perp$ with respect to $\Bl$. $P$ is called \textit{Lagrangian} if $P=P^\perp$.

\begin{definition}\cite{CHL09}
  Let $P$ be an isotropic submodule of $\A(K)$ with respect to $\Bl$. Define $\phi_P$ as the quotient map of $\pi_1(M_K)$ by the kernel of the following map:
  \begin{equation*}
    \pi_1^{(1)}(M_K)\rightarrow \frac{\pi_1^{(1)}(M_K)}{\pi_1^{(2)}(M_K)}=H_1(M_K;\Z[t^{\pm 1}])\rightarrow H_1(M_K;\Qt) \rightarrow \A(K)/ P.
  \end{equation*}
  \textit{The first-order signature $\rho^{(1)}(K, P)$} is defined as $\rho(M_K, \phi_P)$.
\end{definition}

\noindent Note that if $K$ is slice, then the kernel of the inclusion induced map on $\A(K)$ from $M_K$ to the slice disk complement is Lagrangian. Cochran, Harvey and Leidy \cite{CHL10} developed this tool to obstruct $K$ being slice in a more convenient way.

\begin{theorem}\cite[Theorem 4.2, 4.4]{COT03} \cite[Theorem 4.2]{CHL10}\label{thm-CHL-style}
If a knot $K$ is slice, then there exists a Lagrangian submodule $P$ of $\A(K)$ such that $\rho^{(1)}(K, P)=0$. For a slice disk $\Delta$, such a submodule $P$ is given by
\begin{equation*}
  P = \ker(\A(K)\rightarrow H_1(B^4-\nu\Delta;\Qt)).
\end{equation*}
\end{theorem}

Cha \cite{Cha07} introduced the notion of complexity for the Alexander module, the Blanchfield form, and the $\rho$-invariants associated in the fashion of \cite{COT03}. We recall his construction and introduce complexity to the first-order signature in \cite{CHL10}.

\begin{definition}
  Suppose that a knot $K$ bounds a slice disk $\Delta$ in a rational ball $V^4$. $K$ is said to be \textit{rationally slice with complexity $c$} if the inclusion of the knot complement into the disk complement induces multiplication by $c$ on the free part of the first homology as:
  \begin{equation*}
    H_1(S^3-K;\Z)\xrightarrow{\times c} H_1(V - \Delta;\Z)/torsion.
  \end{equation*}
  For a closed 3-manifold $M$ bounding a 4-manifold $W$ with $b_1(M) = b_1(W) = 1$, We also say that $M$ \textit{bounds $W$ with complexity $c$} if the inclusion induces multiplication by $c$ on the free part of the first homology as:
  \begin{equation*}
    H_1(M;\Z)\xrightarrow{\times c} H_1(W;\Z)/torsion.
  \end{equation*}
\end{definition}
\noindent Note that $K$ is rationally slice with complexity $c$ if and only if the $0$-surgery $M_K$ bounds some $\Q HS^1\times B^3$ $W$ with complexity $c$ \cite{CFHH13}.

\begin{definition}
  The \textit{Alexander module $\A_c(K)$ with complexity} $c$ is defined as
  \begin{equation*}
    \A_c(K) = \A(K)\otimes_c\Qt,
  \end{equation*} where the tensor product $\otimes_c$ is given by the action of $t\in\Qt$ as the multiplication by $t^c$.
\end{definition}

\noindent Suppose the $0$-surgery $M_K$ along $K$ bounds $W$ with complexity $c$ with $b_1(W) = 1$. Then the following diagram commutes:
\begin{center}
\begin{tikzcd}
  \pi_1(M_K)\ar{r}{\phi_0}\ar{d} &\Z \ar{r}{t\mapsto t^c} &\Z\\
  \pi_1(W)\ar[swap]{urr}{\psi_0}
\end{tikzcd}
\end{center}
where $\phi_0$ is the abelianization and $\psi_0$ is the abelianization modulo torsion.
Thus, we obtain the inclusion induced map $$j_c:\A_c(K)\rightarrow H_1(W;\Qt),$$ where the right one will be denoted by $\A(W)$.
\noindent The Blanchfield form $\Bl_c$ on $\A_c(K)$ naturally extends $\Bl$ on $\A(K)$ in the sense that:
\begin{definition}\cite[Theorem 5.16]{Cha07}
  The \textit{Blanchfield form $\Bl_c$ on} $\A_c(K)$ is defined as:
  \begin{equation*}
    \Bl_c(x\otimes_c f(t), y\otimes_c g(t)) = f(t)\cdot h(\Bl(x, y))\cdot g(t^{-1}),
  \end{equation*}
  where $h:\frac{\Q(t)}{\Qt}\xrightarrow{t\mapsto t^c}\frac{\Q(t)}{\Qt}.$
\end{definition}

\begin{definition}
  Let $P$ be an isotropic submodule of $\A_c(K)$ with respect to $\Bl_c$. Define $\phi_{c,P}$ as the quotient map of $\pi_1(M_K)$ by the kernel of the following map:
  \begin{equation*}
    \pi_1^{(1)}(M_K)\rightarrow \A_c(K)/ P.
  \end{equation*}
  \textit{The first-order signature $\rho_c^{(1)}(K, P)$ with complexity $c$} is defined as $\rho(M_K, \phi_{c, P})$.
\end{definition}
Now we provide a rational slice obstruction as a rational analogue of Theorem \ref{thm-CHL-style}. This theorem was briefly proved in \cite{Lee24}, where it was not directly used there. For the reader's convenience, we give a proof here in detail.
\begin{proof}[Proof of Theorem \ref{thm-D}]
  Let $W$ be the disk complement $V-\nu\Delta$. $W$ is bounded by the $0$-surgery $M$ along $K$ and $P$ is the kernel of the inclusion induced map $j_c$ on the Alexander module with complexity $c$. For a group $G$, recall that \textit{the $n$-th rational derived subgroup} is $G_\Q^{(n)} = \{g\in G_\Q^{(n-1)} | g^r \in [G_\Q^{(n-1)},G_\Q^{(n-1)}] \text{ for some }r \neq 0 \}$ and $G_\Q^{(0)} = G$ \cite{Coc04}. Consider the following commutative diagram.
\begin{center}
\begin{tikzcd}
  &\dfrac{\pi_1(M)^{(1)}}{\pi_1(M)^{(2)}} \ar{dd}{j_*} \ar{rr} \ar{dr} & & \A_c(K)/P \ar[hook]{dd} \\
  & &\dfrac{\pi_1(M)^{(1)}}{\ker\phi_{c, P}}\ar[dashed,swap]{dl}{\alpha} \ar[hook]{ur} &\\
&\dfrac{\pi_1(W)^{(1)}}{\pi_1(W)^{(2)}_\Q} \ar[hook]{rr} & &\A(W)
\end{tikzcd}
\end{center}
Choose an element $g \in \pi_1(M)^{(1)}$ such that $\phi_{c, P}(g) = 0$. Then $j_*(g) = 0$ so that the dashed map $\alpha$ is induced. Moreover, the map is injective by commutativity of the diagram. To claim that $\rho(M_K, \phi_{c, P})$ is the signature defect of $(W, \psi)$, by Lemma \ref{lem-subgp-property}, it is enough to show that $\beta$ is injective in the following diagram:

\begin{center}
\begin{tikzcd}
  \pi_1(M) \ar{d}{j_*} \ar{r}{\phi_{c,P}} & \dfrac{\pi_1(M)}{\ker{\phi_{c,P}}} \ar{d}{\beta} \\
  \pi_1(W) \ar{r}{\psi} & \dfrac{\pi_1(W)}{\pi_1(W)^{(2)}_\Q}
\end{tikzcd}
\end{center}

\noindent where $\psi$ is the obvious quotient map. Consider the following commutative diagram, where rows are exact and all vertical maps are induced by the inclusion.
\begin{center}
\begin{tikzcd}
  0 \ar{r} &\dfrac{\pi_1(M)^{(1)}}{\ker{\phi_{c,P}}}\ar{d}{\alpha} \ar{r} &\dfrac{\pi_1(M)}{\ker{\phi_{c,P}}}\ar{d}{\beta} \ar{r} &\dfrac{\pi_1(M)}{\pi_1(M)^{(1)}}\ar{d}{\gamma} \ar{r} &0\\
0 \ar{r} &\dfrac{\pi_1(W)^{(1)}}{\pi_1(W)^{(2)}_\Q}\ar{r} &\dfrac{\pi_1(W)}{\pi_1(W)^{(2)}_\Q}\ar{r} &\dfrac{\pi_1(W)}{\pi_1(W)^{(1)}}\ar{r} &0
\end{tikzcd}
\end{center}
We have already checked that $\alpha$ is injective. Note that $\gamma$ is the map on the first homology, which is clearly injective. Thus, $\beta$ is injective and the proof is done.
\end{proof}

Now we relate the first-order signature with complexity with the ordinary one without complexity. Regarding $\A_c(K)$ as $\A(K)\otimes_c \Qt$, we obtain the following identification. See also \cite[Section 5.2]{Cha07}.

\begin{lemma}\label{lem-rho-with-complexity}
  Suppose $P_c$ is an isotropic submodule of $\A_c(K)$ with respect to $\Bl_c$. Let $P$ be the preimage of $P_c$ under $\A(K)\xrightarrow{\otimes_c 1} \A_c(K)$. Then $P$ is an isotropic submodule of $\A(K)$ with respect to $\Bl$. Moreover,
  \begin{equation*}
    \rho_c^{(1)}(K, P_c) = \rho^{(1)}(K, P).
  \end{equation*}
\end{lemma}
\begin{proof}
  Let $x$ be an element in $P$. For any $y\in P$, it is clear that $x\otimes_c 1 \in P_c$ and $y\otimes_c 1 \in P_c \subset P_c^\perp$. Then $\Bl_c(x\otimes_c 1, y\otimes_c 1) = 0$. By definition, we obtain the following identity:
  \begin{align*}
    \Bl_c(x\otimes_c 1, y\otimes_c 1) = 1\cdot h(\Bl(x, y))\cdot 1 = h(\Bl(x, y)).
  \end{align*}
  Since $h$ is clearly injective, $\Bl(x, y) = 0$. Thus, $P\subset P^\perp$ with respect to $\Bl$.

  We now prove that two $\rho$-invariants have the same value. Since the map $\otimes_c 1$ is obviously injective, $P$ is exactly the kernel of $\A(K)\xrightarrow{\otimes_c 1} \A_c(K)\rightarrow \A_c(K)/P_c$ by definition of $P$. Then we have the induced map $\overline{h}:\A(K)/P \rightarrow \A_c(K)/P_c$, which is injective. Thus, by Lemma \ref{lem-subgp-property},
    $$\rho^{(1)}(K, P) = \rho (M_K, \phi_P) = \rho (M_K, \phi_{c, P_c}) = \rho_c^{(1)} (K, P_c).$$
\end{proof}

\noindent One can also see that $\otimes_c 1$ is the same as the map $t\mapsto t^c$ regarding $\A_c(K)$ as $H_1(M_K;\Qt)$ with the local coefficient system from $\pi_1(M_K)\rightarrow \Z\xrightarrow{\times c}\Z$, where the first map is the abelianization.

\begin{remark}
  Note that $P$ in Lemma \ref{lem-rho-with-complexity} might be trivial since $\otimes_c 1$ does not surject in general. For example, let $K$ be the figure-eight knot and $c = 2$. Then the submodule $P_c$ of $\A_c(K)$ annihilated by the factor $t^2-t-1$ of $t^4-3t^2+1$ is isotropic with respect to $\Bl_c$. However, there are no nontrivial isotropic submodules of $\A(K)$ since the Alexander polynomial $t^2-3t+1$ is irreducible. Note that $t^2-t-1\in \A_c(K)$ is not in the image of $\otimes_c 1$ and hence, the preimage $P$ of $P_c$ is 0. See also \cite[p1435]{CHL09}.
\end{remark}

Recall that the first signatures for algebraically slice twist knots $K_n$ with $n\neq 0, 2$ are known to be nonzero as computed in \cite[Example 5.10]{CHL10}. Based on this fact, we can now prove Theorem \ref{thm-AC0} by combining Theorem \ref{thm-D} and Lemma \ref{lem-rho-with-complexity}.
\begin{proof}[Proof of Theorem \ref{thm-AC0}]
  Suppose that $K_n$ is rationally slice with some complexity $c$. Then by Theorem \ref{thm-D}, there exists a Lagrangian submodule $P_c$ of $\A_c(K)$ with respect to $\Bl_c$ such that $\rho_c^{(1)}(K, P_c) = 0$.

  Consider the Alexander polynomial $\Delta_n(t)$ of $K_n$. Note that $\Delta_n(t) = (kt - (k-1))(\left( k-1)t - k \right)$. Since the polynomial $kt - (k - 1)$ is strongly irreducible \cite[Proof of Corollary 3.6]{DPR21}, $P_c$ is annihilated by either $kt^c - (k-1)$ or $(k-1)t^c - k$. Let $P_c^+$ be the Lagrangian annihilated by the former, and let $P_c^-$ be the Lagrangian annihilated by the latter. Let $P^\pm$ be the preimage of $P_c^\pm$ under $\otimes_c 1$. Then $P^\pm$ is annihilated by $kt - (k-1)$ and $(k-1)t - k$, respectively. By Lemma \ref{lem-rho-with-complexity}, $\rho_c^{(1)}(K, P_c^\pm)$ is the same as $\rho^{(1)}(K, P^\pm)$, which can be computed by integrating the Levine-Tristram signature of the torus knot $-T_{k, k- 1}$ as \cite[Example 5.10]{CHL10}. Such values are nonzero whenever $n$ is neither $0$ nor $2$, which contradicts our assumption. Therefore, when $n = k (k-1)$, then $K_n$ is not rationally slice with any complexity $c$ unless $n=0$ or $2$.
\end{proof}

Now it remains to prove the last case, Theorem \ref{thm-AC2}. Before we prove it, we first state a rational slice obstruction developed in \cite{CHL09}. Note that the difference of the following theorem from Theorem \ref{thm-CHL-style} is that $P$ may not be a Lagrangian. This can be simply recovered from Theorem \ref{thm-D} and Lemma \ref{lem-rho-with-complexity}.

\begin{theorem}\cite[Proposition 5.8]{CHL09}\label{thm-CHL-Q}
  If $K$ is rationally slice, then one of the first order $L^2$-signatures of $K$ vanishes. In other words, there exists an isotropic submodule $P$ of $\A(K)$ with respect to $\Bl$ such that $\rho^{(1)}(K, P)=0$.
\end{theorem}

\begin{proof}[Proof of Theorem \ref{thm-AC2}]
  Note that all the twist knots are genus one knots so that any isotropic submodule of $\A(K_n)$ has $\rk_\Q \le 1$. However, $K_n$ in the assumption is not algebraically slice. Thus, the only possible isotropic submodule of $\A(K_n)$ with respect to $\Bl$ is the trivial one $P = 0$. Thus, to use Theorem \ref{thm-CHL-Q}, it is enough to show that $\rho^{(1)}(K_n, 0) \neq 0$. Davis \cite{Dav12a, Dav12b} proved that the values for $k > 8$ does not vanish by approximating them from the Cimasoni-Florens signature \cite{CF08} of derivative links of $K_n \# K_n$, based on \cite[Proposition 5.7]{CHL10}. Therefore, for $n=k^2 > 64$, $K_n$ is not rationally slice.
\end{proof}
\bibliographystyle{amsalpha}
\bibliography{ref}

\providecommand{\bysame}{\leavevmode\hbox to3em{\hrulefill}\thinspace}
\providecommand{\MR}{\relax\ifhmode\unskip\space\fi MR }
\providecommand{\MRhref}[2]{%
  \href{http://www.ams.org/mathscinet-getitem?mr=#1}{#2}
}
\providecommand{\href}[2]{#2}
\begin{thebibliography}{AKPR21}

\bibitem[AGLL20]{AGLL20}
Paolo Aceto, Marco Golla, Kyle Larson, and Ana~G Lecuona, \emph{Surgeries on
  torus knots, rational balls, and cabling}, preprint arXiv:2008.06760 (2020),
  To appear in Ann. Inst. Fourier (Grenoble).

\bibitem[AKPR21]{AKPR21}
Paolo Aceto, Min~Hoon Kim, JungHwan Park, and Arunima Ray, \emph{Pretzel links,
  mutation, and the slice-ribbon conjecture}, Math. Res. Lett. \textbf{28}
  (2021), no.~4, 945--966. \MR{4344693}

\bibitem[AL18]{AL18}
Selman Akbulut and Kyle Larson, \emph{Brieskorn spheres bounding rational
  balls}, Proc. Amer. Math. Soc. \textbf{146} (2018), no.~4, 1817--1824.
  \MR{3754363}

\bibitem[AMP22]{AMP22}
Paolo Aceto, Duncan McCoy, and JungHwan Park, \emph{Definite fillings of lens
  spaces}, preprint arXiv:2208.02586 (2022), To appear in J. Differential Geom.

\bibitem[AMP24]{AMP24}
\bysame, \emph{A survey on embeddings of 3-manifolds in definite 4-manifolds},
  preprint arXiv:2407.03692 (2024).

\bibitem[AS68]{AS68}
M.~F. Atiyah and I.~M. Singer, \emph{The index of elliptic operators. {III}},
  Ann. of Math. (2) \textbf{87} (1968), 546--604. \MR{236952}

\bibitem[BD12]{BD12}
Evan~M. Bullock and Christopher~W. Davis, \emph{Strong coprimality and strong
  irreducibility of {A}lexander polynomials}, Topology Appl. \textbf{159}
  (2012), no.~1, 133--143. \MR{2852954}

\bibitem[CF08]{CF08}
David Cimasoni and Vincent Florens, \emph{Generalized {S}eifert surfaces and
  signatures of colored links}, Trans. Amer. Math. Soc. \textbf{360} (2008),
  no.~3, 1223--1264. \MR{2357695}

\bibitem[CFHH13]{CFHH13}
Tim~D. Cochran, Bridget~D. Franklin, Matthew Hedden, and Peter~D. Horn,
  \emph{Knot concordance and homology cobordism}, Proc. Amer. Math. Soc.
  \textbf{141} (2013), no.~6, 2193--2208. \MR{3034445}

\bibitem[CG78]{CG78}
Andrew~J. Casson and Cameron~McA. Gordon, \emph{On slice knots in dimension
  three}, Algebraic and geometric topology ({P}roc. {S}ympos. {P}ure {M}ath.,
  {S}tanford {U}niv., {S}tanford, {C}alif., 1976), {P}art 2, Proc. Sympos. Pure
  Math., vol. XXXII, Amer. Math. Soc., Providence, RI, 1978, pp.~39--53.
  \MR{520521}

\bibitem[CG85]{ChG85}
Jeff Cheeger and Mikhael Gromov, \emph{Bounds on the von {N}eumann dimension of
  {$L^2$}-cohomology and the {G}auss-{B}onnet theorem for open manifolds}, J.
  Differential Geom. \textbf{21} (1985), no.~1, 1--34. \MR{806699}

\bibitem[CG86]{CG86}
Andrew~J. Casson and Cameron~McA. Gordon, \emph{Cobordism of classical knots},
  \`{A} la recherche de la topologie perdue, Progr. Math., vol.~62,
  Birkh\"{a}user Boston, Boston, MA, 1986, With an appendix by P. M. Gilmer,
  pp.~181--199. \MR{900252}

\bibitem[Cha07]{Cha07}
Jae~Choon Cha, \emph{The structure of the rational concordance group of knots},
  Mem. Amer. Math. Soc. \textbf{189} (2007), no.~885, x+95. \MR{2343079}

\bibitem[Cha08]{Cha08}
\bysame, \emph{Topological minimal genus and {$L^2$}-signatures}, Algebr. Geom.
  Topol. \textbf{8} (2008), no.~2, 885--909. \MR{2443100}

\bibitem[CHH13]{CHH13}
Tim~D. Cochran, Shelly~L. Harvey, and Peter~D. Horn, \emph{Filtering smooth
  concordance classes of topologically slice knots}, Geom. Topol. \textbf{17}
  (2013), no.~4, 2103--2162. \MR{3109864}

\bibitem[CHL09]{CHL09}
Tim~D. Cochran, Shelly~L. Harvey, and Constance Leidy, \emph{Knot concordance
  and higher-order {B}lanchfield duality}, Geom. Topol. \textbf{13} (2009),
  no.~3, 1419--1482. \MR{2496049}

\bibitem[CHL10]{CHL10}
\bysame, \emph{Derivatives of knots and second-order signatures}, Algebr. Geom.
  Topol. \textbf{10} (2010), no.~2, 739--787. \MR{2606799}

\bibitem[CK02]{CK02}
Jae~Choon Cha and Ki~Hyoung Ko, \emph{Signatures of links in rational homology
  spheres}, Topology \textbf{41} (2002), no.~6, 1161--1182. \MR{1923217}

\bibitem[CK21]{CK21}
Jae~Choon Cha and Min~Hoon Kim, \emph{The bipolar filtration of topologically
  slice knots}, Adv. Math. \textbf{388} (2021), Paper No. 107868, 32.
  \MR{4283760}

\bibitem[CO93]{CO93}
Tim~D. Cochran and Kent~E. Orr, \emph{Not all links are concordant to boundary
  links}, Ann. of Math. (2) \textbf{138} (1993), no.~3, 519--554. \MR{1247992}

\bibitem[Coc04]{Coc04}
Tim~D. Cochran, \emph{Noncommutative knot theory}, Algebr. Geom. Topol.
  \textbf{4} (2004), 347--398. \MR{2077670}

\bibitem[COT03]{COT03}
Tim~D. Cochran, Kent~E. Orr, and Peter Teichner, \emph{Knot concordance,
  {W}hitney towers and {$L^2$}-signatures}, Ann. of Math. (2) \textbf{157}
  (2003), no.~2, 433--519. \MR{1973052}

\bibitem[CW03]{CW03}
Stanley Chang and Shmuel Weinberger, \emph{On invariants of {H}irzebruch and
  {C}heeger-{G}romov}, Geom. Topol. \textbf{7} (2003), 311--319. \MR{1988288}

\bibitem[Dav12a]{Dav12a}
Christopher~W. Davis, \emph{Computing the rho-invariants of links via the
  signature of colored links with applications to the linear independence of
  twist knots}, preprint arXiv:1201.6068 (2012).

\bibitem[Dav12b]{Dav12b}
\bysame, \emph{Von {N}eumann rho invariants and torsion in the topological knot
  concordance group}, Algebr. Geom. Topol. \textbf{12} (2012), no.~2, 753--789.
  \MR{2914617}

\bibitem[DHST21]{DHST21}
Irving Dai, Jennifer Hom, Matthew Stoffregen, and Linh Truong, \emph{More
  concordance homomorphisms from knot {F}loer homology}, Geom. Topol.
  \textbf{25} (2021), no.~1, 275--338. \MR{4226231}

\bibitem[Don87]{Don87}
S.~K. Donaldson, \emph{The orientation of {Y}ang-{M}ills moduli spaces and
  {$4$}-manifold topology}, J. Differential Geom. \textbf{26} (1987), no.~3,
  397--428. \MR{910015}

\bibitem[DPR21]{DPR21}
Christopher~W. Davis, JungHwan Park, and Arunima Ray, \emph{Linear independence
  of cables in the knot concordance group}, Trans. Amer. Math. Soc.
  \textbf{374} (2021), no.~6, 4449--4479. \MR{4251235}

\bibitem[Fic84]{Fic84}
Henry~Clay Fickle, \emph{Knots, {${\bf Z}$}-homology {$3$}-spheres and
  contractible {$4$}-manifolds}, Houston J. Math. \textbf{10} (1984), no.~4,
  467--493. \MR{774711}

\bibitem[FM16]{FM16}
Peter Feller and Duncan McCoy, \emph{On 2-bridge knots with differing smooth
  and topological slice genera}, Proc. Amer. Math. Soc. \textbf{144} (2016),
  no.~12, 5435--5442. \MR{3556284}

\bibitem[FQ90]{FQ90}
Michael~H. Freedman and Frank Quinn, \emph{Topology of 4-manifolds}, Princeton
  Mathematical Series, vol.~39, Princeton University Press, Princeton, NJ,
  1990. \MR{1201584}

\bibitem[Fre82]{Fre82}
Michael~H. Freedman, \emph{The topology of four-dimensional manifolds}, J.
  Differential Geom. \textbf{17} (1982), no.~3, 357--453. \MR{679066}

\bibitem[FS84]{FS84}
Ronald Fintushel and Ronald~J. Stern, \emph{A {$\mu$}-invariant one homology
  {$3$}-sphere that bounds an orientable rational ball}, Four-manifold theory
  ({D}urham, {N}.{H}., 1982), Contemp. Math., vol.~35, Amer. Math. Soc.,
  Providence, RI, 1984, pp.~265--268. \MR{780582}

\bibitem[GJ11]{GJ11}
Joshua Greene and Stanislav Jabuka, \emph{The slice-ribbon conjecture for
  3-stranded pretzel knots}, Amer. J. Math. \textbf{133} (2011), no.~3,
  555--580. \MR{2808326}

\bibitem[GO25]{GO25}
Joshua~Evan Greene and Brendan Owens, \emph{Alternating links, rational balls,
  and cube tilings}, Journal of the European Mathematical Society (2025).

\bibitem[HKL16]{HKL16}
Matthew Hedden, Se-Goo Kim, and Charles Livingston, \emph{Topologically slice
  knots of smooth concordance order two}, J. Differential Geom. \textbf{102}
  (2016), no.~3, 353--393. \MR{3466802}

\bibitem[Hom14]{Hom14}
Jennifer Hom, \emph{The knot {F}loer complex and the smooth concordance group},
  Comment. Math. Helv. \textbf{89} (2014), no.~3, 537--570. \MR{3260841}

\bibitem[HW16]{HW16}
Jennifer Hom and Zhongtao Wu, \emph{Four-ball genus bounds and a refinement of
  the {O}zsv\'{a}th-{S}zab\'{o} tau invariant}, J. Symplectic Geom. \textbf{14}
  (2016), no.~1, 305--323. \MR{3523259}

\bibitem[Ker65]{Ker65}
Michel~A. Kervaire, \emph{Les n\oe uds de dimensions sup\'erieures}, Bull. Soc.
  Math. France \textbf{93} (1965), 225--271. \MR{189052}

\bibitem[Kim23]{Kim23}
Taehee Kim, \emph{Knot reversal and rational concordance}, Bull. Lond. Math.
  Soc. \textbf{55} (2023), no.~3, 1210--1221. \MR{4599109}

\bibitem[Kir97]{Kir97}
Robion~C. Kirby, \emph{Problems in low-dimensional topology}, Proceedings 1993
  Georgia International Topology Confence, Geometric Topology \textbf{2}
  (1997).

\bibitem[KL99]{KL99}
Paul Kirk and Charles Livingston, \emph{Twisted {A}lexander invariants,
  {R}eidemeister torsion, and {C}asson-{G}ordon invariants}, Topology
  \textbf{38} (1999), no.~3, 635--661. \MR{1670420}

\bibitem[Lec12]{Lec12}
Ana~G. Lecuona, \emph{On the slice-ribbon conjecture for {M}ontesinos knots},
  Trans. Amer. Math. Soc. \textbf{364} (2012), no.~1, 233--285. \MR{2833583}

\bibitem[Lec15]{Lec15}
\bysame, \emph{On the slice-ribbon conjecture for pretzel knots}, Algebr. Geom.
  Topol. \textbf{15} (2015), no.~4, 2133--2173. \MR{3402337}

\bibitem[Lee24]{Lee24}
Jaewon Lee, \emph{Obstructing two-torsion in the rational knot concordance
  group}, preprint arXiv:2406.12761 (2024).

\bibitem[Lev69]{Lev69}
Jerome Levine, \emph{Invariants of knot cobordism}, Invent. Math. \textbf{8}
  (1969), 98--110; addendum, ibid. {\bf 8 (1969), 355}. \MR{253348}

\bibitem[Lis07a]{Lis07a}
Paolo Lisca, \emph{Lens spaces, rational balls and the ribbon conjecture},
  Geom. Topol. \textbf{11} (2007), 429--472. \MR{2302495}

\bibitem[Lis07b]{Lis07b}
\bysame, \emph{Sums of lens spaces bounding rational balls}, Algebr. Geom.
  Topol. \textbf{7} (2007), 2141--2164. \MR{2366190}

\bibitem[Lis17]{Lis17}
\bysame, \emph{On 3-braid knots of finite concordance order}, Trans. Amer.
  Math. Soc. \textbf{369} (2017), no.~7, 5087--5112. \MR{3632561}

\bibitem[L{\"{u}}c02]{Luc02}
Wolfgang L{\"{u}}ck, \emph{{$L^2$}-invariants of regular coverings of compact
  manifolds and {CW}-complexes}, Handbook of geometric topology, North-Holland,
  Amsterdam, 2002, pp.~735--817. \MR{1886681}

\bibitem[Mil68]{Mil68}
John~W. Milnor, \emph{Infinite cyclic coverings}, Conference on the {T}opology
  of {M}anifolds ({M}ichigan {S}tate {U}niv., {E}. {L}ansing, {M}ich., 1967),
  The Prindle, Weber \& Schmidt Complementary Series in Mathematics, vol. Vol.
  13, Prindle, Weber \& Schmidt, Boston, Mass.-London-Sydney, 1968,
  pp.~115--133. \MR{242163}

\bibitem[MY87]{MY87}
Yukio Matsumoto and Akio Yamada, \emph{Determination of sliceness of the doubly
  twisted knots}, Proceeding to Intelligence of Low-dimensional Topology at
  RIMS Kokyuroku \textbf{624} (1987).

\bibitem[OS03a]{OS03a}
Peter Ozsv\'ath and Zolt\'an Szab\'o, \emph{Absolutely graded {F}loer
  homologies and intersection forms for four-manifolds with boundary}, Adv.
  Math. \textbf{173} (2003), no.~2, 179--261. \MR{1957829}

\bibitem[OS03b]{OS03b}
Peter~S. Ozsv\'{a}th and Zolt\'{a}n Szab\'{o}, \emph{Knot {F}loer homology and
  the four-ball genus}, Geom. Topol. \textbf{7} (2003), 615--639. \MR{2026543}

\bibitem[OSS17]{OSS17}
Peter~S. Ozsv\'{a}th, Andr\'{a}s~I. Stipsicz, and Zolt\'{a}n Szab\'{o},
  \emph{Concordance homomorphisms from knot {F}loer homology}, Adv. Math.
  \textbf{315} (2017), 366--426. \MR{3667589}

\bibitem[Ras03]{Ras03}
Jacob~A. Rasmussen, \emph{Floer homology and knot complements}, ProQuest LLC,
  Ann Arbor, MI, 2003, Thesis (Ph.D.)--Harvard University. \MR{2704683}

\bibitem[{\c{S}}av24]{Sav24}
O\u{g}uz {\c{S}}avk, \emph{A survey of the homology cobordism group}, Bull.
  Amer. Math. Soc. (N.S.) \textbf{61} (2024), no.~1, 119--157. \MR{4678574}

\bibitem[Sie75]{Sie75}
L~Siebenman, \emph{Exercises sur les noeuds rationnels}, mimeographed notes
  (1975).

\bibitem[Sim21]{Sim21}
Jonathan Simone, \emph{Using rational homology circles to construct rational
  homology balls}, Topology Appl. \textbf{291} (2021), Paper No. 107626, 16.
  \MR{4215138}

\bibitem[Sim23]{Sim23}
\bysame, \emph{Classification of torus bundles that bound rational homology
  circles}, Algebr. Geom. Topol. \textbf{23} (2023), no.~6, 2449--2518.
  \MR{4640131}

\bibitem[Vir73]{Vir73}
O.~Ja. Viro, \emph{Branched coverings of manifolds with boundary, and
  invariants of links. {I}}, Izv. Akad. Nauk SSSR Ser. Mat. \textbf{37} (1973),
  1241--1258. \MR{370605}

\end{thebibliography}
\end{document}